\documentclass[reqno,11pt]{amsart}
\usepackage{amscd,amsfonts,amsmath,amssymb,amstext,amsthm}
\usepackage[utf8]{inputenc}
\usepackage[T2A]{fontenc}
\usepackage{cite}
\usepackage[unicode]{hyperref}
\usepackage{color}
\usepackage{xcolor}
\usepackage{comment}
\usepackage{tikz}
\usepackage{tikz-cd}
\usepackage[normalem]{ulem}

\usetikzlibrary{matrix,arrows,decorations.pathmorphing}

\makeatletter
\def\@settitle{\begin{center}%
    \baselineskip14\p@\relax
    \bfseries
    \MakeUppercase{\@title}
  \end{center}%
}
\makeatother

\newtheorem{theorem}{Theorem}

\newtheorem{lemma}{Lemma}
\numberwithin{lemma}{section}
\newtheorem{proposition}{Proposition}
\newtheorem{corollary}{Corollary}
\theoremstyle{remark}

\newtheorem{example}{Example}
\theoremstyle{definition}
\newtheorem{definition}{Definition}
\numberwithin{equation}{section}

\def\C{{\mathbb C}}
\def\Q{{\mathbb Q}}
\def\Cn*{{\C^n}^*}
\def\T{{\mathbb T}}
\def\L{{\mathbb L}}

\def\S{{\mathbb S}}

\def\rank{{\rm rank\:}}
\def\R{{\mathbb R}}
\def\Z{{\mathbb Z}}
\def\codim{{\rm codim\:}}
\def\codima{{\rm codim_a}}

\def\vol{{\rm vol}}
\def\ES{{\rm ES}}
\def\EAS{{\rm EAS}}

\def\rfm{{\rm RFD}}

\def\E{\:{\rm e}}
\def\re{{\rm Re\:}}
\def\im{{\rm Im\:}}
\def\supp{{\rm supp\:}}

\newcounter{par}
\newcounter{spr}
\begin{document}
\title{
Ring of conditions for
${\mathbb C}^\MakeLowercase{n}$}
\author{Boris Kazarnovskii}
\address {Institute for Information Transmission Problems \newline
{\it kazbori@gmail.com}.}
\begin{abstract}
The exponential sum (ES)
is a linear combination of characters of an additive group $\mathbb C^n$.
The exponential analytic set (EAS) is a set of common zeroes of
a finite tuple of ESs.
We consider ES and EAS as an analogs of Laurent polynomial
and of algebraic variety in complex torus $(\C\setminus0)^n$.
Respectively we construct the ring of conditions for $\mathbb C^n$ as
an analog of the ring of conditions for $(\C\setminus0)^n$.
The construction of this ring is based on the definition of associated  to \EAS\
algebraic subvariety of some multidimensional torus
and on the applying tropical algebraic geometry to this subvariety.
Just as in the case of a torus,
the ring of conditions is generated by hypersurfaces.
This preprint is 
an extended summary of the article
proposed to ''Izvestiya: Mathematics''.
\end{abstract}
\maketitle
\section{\ES s, \EAS s, windings and models}
\label{EAS s}
\emph{Exponential sum} (\ES) is a function
in $\C^n$ of the form
$$f(z)=\sum_{\lambda\in\Lambda\subset{\C^n}^*,\:c_\lambda\in\C}c_\lambda \E^{\langle z,\lambda\rangle},$$
where $\Lambda$ is a finite subset of the dual space ${\C^n}^*$.
The set $\Lambda$ is called the \emph{support} of \ES.
The \emph{Newton polytope} of \ES\ is a convex hull of its support.
If $\Lambda\subset\re{\C^n}^*$ then we say that $f$ is a \emph{quasialgebraic} \ES.
The set of common zeroes of a finite tuple of (quasialgebraic) \ES s
is called (quasialgebraic) \emph{exponential analytic set} (\EAS).
The ring of \ES s consists of linear combinations of
characters of the additive group $\C^n$,
i.e. this ring is similarly to the ring of Laurent polynomials,
which are linear combinations of characters of a torus $(\C\setminus0)^n$.
Focusing on this similarity,
we construct \emph{the ring of conditions}
for the inter\-sec\-tion theory for quasialgebraic \EAS s.
The construction is based on the definition of
some algebraic variety,
associated  to \EAS,
and on the applying tropical algebraic  geometry. 
The intersection theory for arbitrary \EAS s will be described in the following publications.

Let $G\subset\C^{n*}_+$ be a finitely generated subgroup.
Suppose that $G$ contains some basis of a dual space ${\C^n}^*$.
Let $\T_G=(\C\setminus0)^q$ where $q=\rank G$.
Choosing a basis $\lambda_1,\ldots,\lambda_q$ of $G$
we consider the homomorphism
$$\omega_G\colon z\mapsto(\E^{\langle z,\lambda_1\rangle},\ldots,\E^{\langle z,\lambda_q\rangle})\in\T_G$$
of the additive group $\C^n _+$ to the torus $\T_G$.
%
\begin{definition} \label{dfWinding}
The image $\omega_G(\C^n)$
is called 
\emph{the standard winding} of the torus $\T_G$,
and the mapping $\omega_G\colon{\C^n}\to\T$ itself
is called \emph{the mapping of standard winding}.
If $\omega_G(\im\C^n)$ is contained in the compact subtorus of the torus $\T_G$
(i.e. if $G\subset{\rm Re}\:{\C^n}^*$),
then we say that the winding is quasialgebraic.
\end{definition}
\begin{corollary} \label{corCharTorus}
The torus $\T_G$ is a torus of characters of the group $G$.
The invariant definition of the mapping of standard winding
is
$\omega_G(z)(g)=\E^{\langle z,g\rangle}$.
\end{corollary}
\begin{corollary} \label{corWinding1}
The standard winding is dense in Zariski topology.
\end{corollary}
\begin{corollary} \label{corWinding2}
Let $E_G$ be a ring of \ES {\rm s} with supports in $G$.
Then the mapping
$\omega_G^*\colon\C[\T_G]\to E_G$
is an isomorphism of the rings.
\end{corollary}
\begin{definition} \label{dfModel}
Let $X$ be an \EAS\ with ideal of equations $I\subset E_G$.
Keeping the notation $I$ for the corresponding ideal of $\C[\T_G]$
we say that the zero variety $M_G\subset\T_G$ of $I$
is a \emph{model} of \EAS\ $X$.
\end{definition}
\begin{definition} \label{dfAlg1}
If the variety $M_G$ is equidimensional,
then \EAS\ $X$ is also called \emph {equidimensional}.
Denote $\codim M_G$ by $\codima X$ and
call \emph{algebraic codimension} of \EAS\ $X$.
\end{definition}
 %
 Note that both the equidimensionality and the algebraic codimension of $X$
 do not depend on the choice of the group $G$,
such that the equations of \EAS\ $\:X$ are contained in the ring $E_G$.
Any \EAS\ is a union of a finite number of equidimensional \EAS s
of different algebraic codimensions.
Next, by default,
any \EAS\ is  assumed to be equidimensional.
Equidimensional \EAS\ $X$ of algebraic codimension $n$
is an analogue of a zero-dimensional algebraic variety.
The set of points of such \EAS\ is infinite.
For example, the set of zeros of the function $\E^z-1$ in $\C^1 $ is $ 2\pi i\:\Z$.
An analogue of the number of points of zero-dimensional algebraic variety
is a \emph{weak density} $d_w(X)$,
see the Definition \ref{dfWeakDensity}.
\begin{example}[см. \cite{K97, Z02, BMZ07}]\label{exAnomal}
Let $X$ be an \EAS\ with the equations $f = g = 0$, $f,g\in E_G$.
If $f,g$ have no common divisor in the ring $E_G$,
then $\codima X = 2$,
otherwise $\codima X = 1$.
For example,
the algebraic codimension of the point $0\in\C$,
considered as the \EAS\ with the equations
$\E^z-1 = \E^{\sqrt{2}z} -1 = 0$,
is equal to $2$.
Therefore,
the codimension of \EAS, as an analytic set,
can be less than $\codima X$.
Let $(X,z)$ be an irreducible germ of \EAS\ $X$ at
$z\in X$.
If the codimension of $(X,z)$ less than $\codima X$,
then the germ is said to be \emph{atypical}.
It is known, that any atypical germ of \EAS\
belongs to some
proper affine subspace of $\C^n$.
In particular,
any
atypical component of  \EAS\ of algebraic
codimension $2$ in $\C^2$ is an affine line.
\end{example}
\section{Tropicalization and intersection index of \EAS s}\label{Index}
Next, we consider quasialgebraic \ES s only.
We use the 
notations:
\par\smallskip

$\mathcal T_G$ is a Lie algebra of $\T_G$

$V=\re\mathcal T_G$ is the space of one-parameter subgroups $\T_G$, $N=\dim V$

$\R^n=\re\C^n$

$s_G\colon\R^n\to\re\mathcal T_G$ is
a restriction of differential $d\omega_G$ to the space $\R^n$.
\par\smallskip
Here we define the tropicalization of \EAS.
Its independence from the choice of the group $G$ is proved in the section \ref{Equasi}.
The tropical notions used below are defined in Section \ref{pullBack}.

Let $\mathcal K\subset\re\mathcal T_G$ be a tropical fan of
algebraic variety $M_G\subset\T_G$,
$L_G=s_G(\R^n)\subset\re\mathcal T_G$,
and let $\mathcal L_G$  be a tropical fan
consisting of a cone $L_G$ with the Euclidian weight $w(L_G)=1$;
see subsection \ref{tropVar}.
Consider the product $\mathcal L_G\cdot\mathcal K$
of Euclidian tropical varieties.
The support of a tropical fan $\mathcal L_G\cdot\mathcal K$ is contained in $L_G$.

Let $s\colon U\to V$ be a linear operator.
In Subsection \ref{pullBacks},
for any tropical variety $\mathcal M\subset V$
we define the tropical variety $s^*(\mathcal M)$ in $U$,
called the pull back of $\mathcal M$.
\begin{definition} \label{dfTropicalizationG}
Let $M_G$ be a model of \EAS\ $X$ and
$\mathcal K$ be a tropicalization of algebraic variety $M_G$.
The tropical variety $X^{\rm trop}=s_G^*(\mathcal L_G\cdot\mathcal K)\subset\R^n$
is called a \emph{tropicalization of} \EAS\ $X$.
\end{definition}
Let $\Delta\subset{\R^n}^*$ be a convex polyhedron
and let $K_\Lambda\subset\R^n$ be a dual cone of a face $\Lambda$ of $\Delta$.
The set of cones $\mathcal K_{\Delta,k}=\{K_\Lambda\colon\:\Lambda\subset\Delta,\:\dim\Lambda\geq k\}$
form the $(n-k)$-dimensional fan of cones.
For $\dim\Lambda=k$ we put $w(K_\Lambda)$ equal to
$k$-dimensional volume $\vol_k(\Lambda)$ of $\Lambda$.
Then  $\mathcal K_{\Delta,k}$ is a Euclidean tropical fan.
%
\begin{corollary}
Let $X=\{z\in\C^n\colon\:f(z)=0\}$ be a quasialgebraic exponential hypersurface.
Then $X^{\rm trop}=\mathcal K_{\Delta,1}$
where $\Delta$ is a Newton polyhedron of \ES\ $f$.
\end{corollary}
\begin{proof}
Let $\Gamma$ be a Newton polyhedron of Laurent polynomial $F=\omega^*_G(f)$.
Then the tropicalization of algebraic variety $M_G=\{g\in\T_G\colon F(g)=0\}$
is a fan $\mathcal K_{\Gamma,1}\subset  V$ consisting of cones, dual to faces of $\Gamma$.
Note that $s'_G\Gamma=\Delta$,
where $s'_G\colon(\re\mathcal T_G)^*\to{\R^n}^*$ is a linear operator adjoint to $s_G$.
Now the statement follows from Theorem \ref{thmBack} (3);
see also \hyperlink{[4]}{[{\bf 4}]}.
\end{proof}
If $\codima X=n$,
then $\dim X^{\rm trop}=0$.
In this case, denote by $d_w(X)$ the weight of the zero cone in $X^{\rm trop}$.
\begin{definition} \label{dfWeakDensity}
We call $d_w(X)$ \emph{a weak density} of \EAS\ $X$.
\end{definition}
\begin{definition} \label{dfIndex}
Let $X_1,\ldots,X_k$ be \EAS s of
total algebraic codimension $n$.
The weight of a zero cone in a $0$-dimensional tropical fan
$X^{\rm trop}_1\cdot\ldots\cdot X^{\rm trop}_k $
is denoted by $I(X_1,\ldots,X_k)$ and called the intersection index
of \EAS s $X_1,\ldots,X_k$.
\end{definition}
The following statement is a quasialgebraic analogue of BKK theorem; see also \cite{K81, Few}.
\begin{theorem} \label{thmBKK}
Let $X_1\ldots,X_n$ be zero surfaces of quasialgebraic \EAS {\rm s} $f_1,\ldots,f_n$
with Newton polyhedra $\Delta_1,\ldots,\Delta_n$.
Then
$$
I(X_1,\ldots,X_n)=\frac{n!}{(2\pi)^n}V_n(\Delta_1,\ldots,\Delta_n),
$$
where $V_n(\Delta_1,\ldots,\Delta_n)$ is a mixed volume of polyhedra.
\end{theorem}
\begin{proof}
Follows from the tropical BKK theorem,
see, for example,
\cite{EKK20} [Theorem 3.1.3].
\end{proof}
\section{Ring of conditions $\mathcal E_G$}
\label{EG}
Let $\mathcal V(E)$ and $\mathcal V(V)$ be the rings of Euclidean tropical varieties
in $E$ and $V$,
$s\colon E\to V$ be a linear operator.
In Subsection \ref{pullBacks}, the mapping of an inverse image
$s^*\colon\mathcal V(V)\to\mathcal V(E)$ is defined.
Below we use the following properties of this mapping:
\par\smallskip
\noindent
\hypertarget{[1]}{[{\bf 1}]}
$s^*$ is a homomorphism of $\R$-algebras

\noindent
\hypertarget{[2]}{[{\bf 2}]}
$s^*$ preserves codimensions of varieties

\noindent
\hypertarget{[3]}{[{\bf 3}]}
For any linear operator $u\colon Z\to E$ we have $(s\cdot u)^*=u^*\cdot s^*$

\noindent
\hypertarget{[4]}{[{\bf 4}]}
Let $s'\colon V^*\to E^*$ be an adjoint to $s$ linear operator
and $\Delta\subset V^*$ be a convex polyhedron.
Then for any $k\leq\dim V$ we have
 $s^*\mathcal K_{\Delta,k}=\mathcal K_{s'\Delta,k}$.

\noindent
\hypertarget{[5]}{[{\bf 5}]}
Let $\rho\colon\T\to\mathbb H$ be a homomorphism of tori
and let $M\subset\mathbb H$ be an algebraic variety.
Denote by $Y_{\rm tr}$ the tropical fan of algebraic variety $Y$.
Then $(\rho^{-1}M)_{\rm tr}=s^*M_{\rm tr}$,
where $s=d\rho\colon\re\mathcal T\to\re\mathcal H$.
\par\smallskip
\begin{lemma} \label{lmTropicalizations1}
Tropicalizations of \EAS {\rm s} in the ring $\mathcal V(\R^n)$
forms a semigroup in addition and a semigroup in multiplication.
\end{lemma}
\begin{proof}
The set of \EAS\ tropicalizations is a set of tropical fans of the form $\{s_G^*\mathcal K\}$,
where $\mathcal K$ is a tropicalization of some algebraic variety.
Therefore, the statement follows from the well-known property of algebraic tropicalizations:
\emph{if $\mathcal P, \mathcal Q$ are tropicalizations of algebraic varieties $P, Q$,
then for a general $g\in\T_G$ the tropical varieties
$\mathcal P + \mathcal Q$ and $\mathcal P\cdot\mathcal Q$
are tropicalizations of respectively $gP\cup Q$ and $gP\cap Q$}.
\end{proof}
\begin{theorem} \label{thmTropicalizations}
Integer linear combinations of \EAS\ tropicalizations
form a subring $\mathcal E_G$ of the ring $\mathcal V(\R^n)$.
The $\Q$-algebra $\mathcal E_G\otimes\Q$ is generated by tropical fans of the form $\mathcal K_{\Delta,1}$,
where $\Delta\subset{\R^n}^*$ is a convex polyhedron with vertices at points of the group $G$.
\end{theorem}
\begin{proof}
Integer linear combinations of tropicalizations of algebraic varieties
form a subring $\mathcal R$ of the ring $\mathcal V(\re\mathcal T_G)$.
In addition,  the $\Q$-algebra $\mathcal R\otimes\Q$
is generated by tropical fans of the form $\mathcal K_{\Lambda,1}$,
where $\Lambda\subset(\re\mathcal T_G)^* $ is a convex polyhedron with vertices at integer points;
see \cite{EKK20}.
The linear operator $s'_G\colon(\re\mathcal T_G)^*\to{\R^n}^*$,
adjoint to $s_G$,
translates integer points of $(\re\mathcal T_G)^*$ into points of the group $G\subset{\R^n}^*$.
Therefore, the statement follows from
the given above properties \hyperlink{[1]}{[{\bf 1}]} and \hyperlink{[4]}{[{\bf 4}]}
of the mapping $s^*_G$.
\end{proof}
%
Let $G$ be an additive subgroup in a real $N$-dimensional vector space $E$.
Now we define \emph{a ring of convex polyhedra with vertices at the points of the group} $G$.
First consider the space $H$ of virtual convex polyhedra with vertices at points of $G$
(recall that a virtual polyhedron is the formal difference of two convex polyhedra).
Let $\mathcal S(H)=\sum_{m\geq0}\mathcal S_m(H)$ be a symmetric algebra of $H$.
For $\mathcal S_N(H)\ni s=\Delta_1\cdot\ldots\cdot\Delta_N$
we set $I(s)$ equal to the mixed volume of
$\Delta_1,\ldots,\Delta_N$.
We associate with the linear functional $I\colon\mathcal S_N(H)\to\R$
the homogeneous ideal $J\subset\mathcal S(H)$
generated by the following sets of generators:
1) $\ker I$, 2) $\sum_{m>N}\mathcal S_m(H)$ and
3) $\{s\in\mathcal S_k(H)\:\vert\:s\cdot\mathcal S_{N-k}(H)\subset\ker I, k=1,\ldots,N-1\}$.
Note that the ideal $J$ does not depend on the choice of Lebesgue measure in the space $E$.
\begin{definition} \label{dfPolRing}
The ring ${\rm Pol}(E;G)=\mathcal S(H)/J$  is called the ring of convex polyhedra
with vertices in $G$.
\end{definition}
\begin{theorem}\label{thmqIsoPol}
The mapping $\Delta\mapsto\mathcal K_{\Delta,1}$
extends to the ring isomorphisms

{\rm(i)}  ${\rm Pol}(E;E)\to\mathcal V(E^*)$ for $G=E$

{\rm(ii)} ${\rm Pol}(E;\Z^N)\otimes\Q\to\mathcal \Q(E^*)$, where  $\mathcal \Q(E^*)$
is a ring of rational tropical varieties (see Definition {\rm\ref{dfEuQ}})

{\rm(iii)}  ${\rm Pol}({\R^n}^*;G)\otimes\Q\to\mathcal E_G$
\end{theorem}
\begin{proof}
The statements (i), (ii) see in \cite{EKK20}.
The statement (iii) follows from Theorem \ref{thmTropicalizations}.
\end{proof}
\begin{corollary} \label{corMainG2}
Let $X,Y$ be equidimensional \EAS {\rm s}, $\codima X=\codima Y=k$.
Then the following conditions are equivalent

{\rm(1)} $X^{\rm trop}=Y^{\rm trop}$

{\rm(2)} for any \EAS\ $\:Z$ of algebraic codimension $n-k$,
$I(X,Z)=I(Y,Z)$
\end{corollary}
\begin{proof}
It follows from Definition \ref{dfIndex} that (1)$\Rightarrow$(2).
On the other hand,
it follows from (2)
that $$\forall \mathcal K\in\mathcal E_G\colon\:(X^{\rm trop}-Y^{\rm trop})\cdot\mathcal K=0.$$
Now the statement follows from the nondegeneracy of pairing in the ring ${\rm Pol}({\R^n}^*;G)$.
This non-degeneracy follows from the non-degeneracy of pairing for tropical varieties (see Theorem \ref{thmPair})
and from the standard properties of the mixed volume of polyhedra.
\end{proof}
\section{Ring of conditions $\mathcal E^{\rm quasi}$}
\label{Equasi}
In this Section we prove that
1) the notion of tropicalization  is independent
from the choice of the ring $E_G$,
containing the equations \EAS,
and 2) any tropical variety in $\R^n$ is a tropicalization of some \EAS.
Thus, the ring $\mathcal E^{\rm quasi}$, formed by tropicalizations of \EAS s,
coincides with the ring of tropical varieties in $\R^n$.
\begin{proposition}\label{prD1}
Let $G\subset H\subset{\R^n}^*$, $\pi_{H, G}\colon\T_H\to\T_G$ -
the character restriction from $H$ to $G$.
Then the mappings $s^*_G,\,s^*_ H\cdot(d\pi_{H, G})^*\colon\mathcal V(\re\mathcal T_G)\to\mathcal V(\R^n)$
coincide.
\end{proposition}
\begin{proof}
By definition $s_G=s_H\cdot d\pi_{H,G}$,
i.e.
the diagram
\begin{equation}\label{eqD0}
 \begin{tikzcd}
&&\re\mathcal T_H\arrow{dd}{d\pi_{H,G}}
\\\R^n\arrow{rru}{s_H}\arrow{rrd}{s_G}\\
&&\re\mathcal T_G
\end{tikzcd}
\end{equation}
%
is comutative.
It follows from \hyperlink{[3]}{[{\bf 3}]} that the dual diagram
 \begin{equation}\label{eqD1}
 \begin{tikzcd}
&&\mathcal V(\re\mathcal T_H)\arrow{lld}{s^*_H}
\\\mathcal V(\R^n)\\
&&\mathcal V(\re\mathcal T_G)\arrow{uu}{(d\pi_{H,G})^*}\arrow{llu}{s^*_G}
\end{tikzcd}
\end{equation}
is also comutative.
\end{proof}
\begin{corollary} \label{corD1}
The tropicalization $X^{\rm trop}$ does not depend on the choice of the group $G$,
containing the equations of \EAS\ X.
\end{corollary}
\begin{proof} Follows from (\ref{eqD1}).
\end{proof}
\begin{definition} \label{dfTropikalization}
A subring of the ring $\mathcal V(\R^n)$,
consisting of \EAS\ tropicalizations,
denote by $\mathcal E^{\rm quasi}$ and call
\emph{quasialgebraic ring of conditions for an affine space}.
\end{definition}
\begin{theorem} \label{thmqRing}
$\mathcal E^{\rm quasi}=\mathcal V(\R^n)$.
\end{theorem}
%
We start with the proof of the following statement.
\begin{lemma}\label{lmSurj}
The ring $\mathcal V(\R^n)$ (regarded as a $\Z$-algebra)
  generated by elements of the form $s^*_G(K_{\Lambda, 1})$,
where $G$ and $\Lambda$ change respectively in the set of
finitely generated subgroups of the space ${\R^n}^*$
and in the set of convex polyhedra with vertices in the character lattice of the torus $\T_G$.
\end{lemma}
\begin{proof}
Let $\Delta\subset{\R^n}^*$ be a convex polyhedron.
By construction,
the linear operator $s'_G\colon(\re\mathcal T_G)^*\to{\R^n}^*$,
adjoint to $s_G$,
takes the points of the character lattice of the torus $\T_G$ to the points of the group $G$.
Choose the group $G$,
containing the vertices of the polytope $\Delta$.
Let $\chi_\delta$ be a character of $\T_G$,
such that $s'_G(\chi_\delta)=\delta$.
We denote by $\Lambda$  the convex hull of the characters $\chi_\delta$.
Then $\Lambda$ is a convex polyhedron in the space $(\re\mathcal T_G)^*$,
and $s'_G(\Lambda)=\Delta$.
From here, according to \hyperlink{[4]}{[{\bf 4}]},
$\mathcal K_{\Delta, 1}=s^*_G(\mathcal K_{\Lambda, 1})$.

Tropical varieties of the form $\mathcal K_{\Delta,1}$
generate the $\R$-algebra $\mathcal V(\R^n)$
(see Theorem \ref{thmTropicalizations}).
Now it remains to notice that
$r\cdot\mathcal K_{\Delta,1}=\mathcal K_{r\cdot\Delta, 1}$
for any real $r>0$.
\end{proof}
\noindent
\emph{Proof of Theorem} \ref{thmqRing}.
It follows from Lemma \ref{lmSurj},
because $\mathcal K_{\Delta, 1}$ is the tropicalization of an exponential hypersurface
with the Newton polyhedron $\Delta$.
\begin{corollary} \label{cortoTrop2}
$\Z$-algebra $\mathcal E^{\rm quasi}$
is generated by the tropicalizations of exponential
hypersurfaces.
\end{corollary}
\begin{proof}
Follows from Lemma \ref{lmSurj}.
\end{proof}
\section{Domains of relatively full measure}\label{Def3}
This section defines some of the notions used in the formulation of \EAS s intersection theorems.
\begin{definition} \label{dfDensity}
Let $Y$ be a subset of a finite-dimensional Euclidean vector space $E$,
 $B_r\subset E$ be a ball of radius $r$ with center at $0$,
 and $\sigma_n$ be a volume of the unit ball in $\R^n$.
 Denote by $N(Y,r)$ the number of isolated points of $Y\cap B_r$.
 If
  $
 \lim_{r\to\infty}\frac{N(Y,r)}{\sigma_n r^n}
 $
 exists, then we denote it by $d_n(Y)$
 and call the $n$-density of $Y$.
 \end{definition}
\begin{example}\label{ex2.2}
If \EAS\ $X\subset\C^1$ is given by the equation $f(z)=0$,
then the $1$-density $d_1(X)$ exists and equals to $\frac{p}{2\pi}$,
where $p$ is a semiperimeter of the Newton polygon of $f$.
If $f(z)=\E^{\alpha z}-c$,
then $d_1(X) = \frac{|\alpha|}{2\pi}$
(the perimeter of the polygon "segment" \ is considered equal to its doubled length).
\end{example}
\begin{definition} \label{dfApproxLattice}
{\rm(1)}
Let $\mathcal Z\subset E$ be a lattice in the space $E$ with an integer positive multiplicity $m(\mathcal Z)$,
and let $X\subset E$ be a set of points with multiplicities.
We call $X$ $\:\varepsilon$-perturbation of the shifted lattice $z+\mathcal Z$,
if  {\rm1)}  $X$ belongs to $\varepsilon$-neighborhood
$(z+\mathcal Z)_\varepsilon$ of this lattice and
{\rm2)} in $\varepsilon$-neighborhood of any point $x\in z +\mathcal Z$
contains exactly $m(\mathcal Z)$ points $X$.

{\rm(2)}
If the sets $X_1,\ldots,X_m$ are $\varepsilon$-perturbations of
shifted lattices $z_j+\mathcal Z_j$,
then the set $\bigcup_{1\leq j\leq m} X_j$ we call
an $\varepsilon$-perturbation of a union of shifted lattices
$\bigcup_{1\leq j\leq m} (z_j + \mathcal Z_j)$.
\end{definition}
\begin{corollary}\label{corpertLattice}
Let $X$ be an $\varepsilon$-perturbation of a union
of $z_1+\mathcal Z_1,\ldots$ and let $\forall j\colon\:\rank Z_j=n$.
Then the $n$-density $d_n(X)$ exists and equal $\sum_j d_n(\mathcal Z_j)$.
\end{corollary}
Let $\mathfrak I = \{I\}$ be a finite set of proper subspaces in
 finite dimensional real vector space $E$ and
 $\,B_{\mathfrak I}=E\setminus\bigcup_{I\in\mathfrak I}I$.
 Denote by $B_{\mathfrak I,1},B_{\mathfrak I, 2},\ldots$ the connected components
 of $B_{\mathfrak I}$.
 For $0<R\in\R$ we denote by $B^R_{\mathfrak I}\subset E$ the set of points
located at a distance $\geq R$ from $\bigcup_{I\in\mathfrak I}I$.
\begin{definition} \label{dfdRelFull}
We say that $U\subset E$
   is a domain of relatively full measure (\rfm),
  if $\mathfrak I$ and $R>0$ exist,
  such,
  that $U\supset B^R_\mathfrak I$.
The set of subspaces $\mathfrak I = \{I \}$
we will call the base of \rfm.
If an integer lattice is given in the space $E$,
and all the subspaces $I\in\mathfrak I$ are rational (i.e., generated by lattice points),
then the base is called rational.
We say that  \rfm\ with a rational base is a rational \rfm.
\end{definition}
\noindent
Let's list some corollaries of Definition \ref{dfdRelFull}.
\begin{corollary}\label{corrfm}
{\rm(1)}
The union and intersection of \rfm {\rm s} is also \rfm.
The rationality property of \rfm\ under unions and intersections of \rfm  {\rm s} is preserved.

{\rm(2)}
The property of the domain to be \rfm\
is independent of the choice of metric in the space $E$.

{\rm(3)}
If the subspace $L\subset E$ does not belong to the base of \rfm\ $U$,
then $U\cap L$ is \rfm\ in the space $L$.

{\rm(4)}
The inverse image of \rfm\ under a surjective linear map
is \rfm.

{\rm(5)} Domains $B^R_{\mathfrak I, i} = B^R_{\mathfrak I} \cap B_{\mathfrak I, i}$
are the connected components of
 \rfm\ $B^R_{\mathfrak I}$.
 \end{corollary}
\section{Intersections of \EAS s}
\label{dens}
Throughout, it is assumed that,
if a finite set of \EAS s and the group $G\subset{\R^n}^ * $ are involved in the statement,
then all these \EAS s \emph {are given by equations from the ring $ E_G $.}
For a fixed $G$
such \EAS s one-to-one correspond to algebraic subvarieties of the torus $\T_G$
(recall that the variety $M_G$ corresponding to \EAS\ $X$ is called the model $X$).
Therefore, the action of the torus on varieties is also defined on \EAS s.
The action of the element $t\in\T_G$ on \EAS\ $ X $ is denoted by $t\colon X\to X^t$.
This action is a ''toric shift'',
those. continuation of the shift action $\C^n$:
if $z\in\C^n$, then $z + X = X^{\omega_G (z)} $.

Let $X$ be an \EAS\ of algebraic codimension $n$.
The following theorem states
that there exists a finite set of proper subspaces $\mathfrak I$
of the space $\re\mathcal T_G$, such that
if $R\gg0$ and ${\rm Re}\log t\in B^R_\mathfrak I$
(see Definition \ref{dfdRelFull})
then the toric shift $X^t$
is a small perturbation of a finite union
of shifted $n$-dimensional lattices,
located in $\im\C^n$.
\begin{theorem}\label{thmq1}
There is $\mathfrak I$ such that to each of the connected components
$B_{\mathfrak I, 1}, B_{\mathfrak I, 2} \ldots$ of the domain $B_\mathfrak I$
there corresponds a finite set of $n$-dimensional lattices
$$
\{\mathcal L_{i,j}\subset\im\C^n\colon j=1,\ldots,N_i\}
$$
(the lattices $\mathcal L_{i,1}, \mathcal L_{i,2},\ldots$ may sometimes match
and equipped with integer positive multiplicities),
such that for $R\gg0$ and ${\rm Re}\log t\in B^R_{\mathfrak I,i}$ we have:
\par\smallskip
{\rm(1)}
\EAS\ $\:X^t$
is an $\varepsilon$-perturbation of the union of shifted lattices
$$z_1(t)+\mathcal L_{i,1},\ldots,z_{N_i}(t)+\mathcal L_{i,N_i},$$
where functions $z_j(t)$ are continuous
and $\varepsilon\to0$ if $R\to\infty$

{\rm(2)}
$n$-density $d_n(X^t)$ independent of the choice $B_{\mathfrak I,i}$.
\end{theorem}
\begin{theorem}\label{thmq2}
Let $\codima X + \codima Y=n$. Then there is a finite set of subspaces $\mathfrak I=\{I\subset\R^n\}$,
such that the following is true.
If $R$ is large enough, then for all $z\in B^R_{\mathfrak I}+\im\C^n\:$
\EAS s $(z+X)\cap Y$ are equidimensional,
  their algebraic codimensions are equal to $n$,
  and weak densities are the same.
\end{theorem}
Apply tropical intersection index properties
leads to the next more familiar in algebra
statement of the previous theorem.
\begin{corollary}\label{corthmq2}
There is a quasi-algebraic exponential hypersurface $Z\subset\C^n$,
such that for $w\not\in Z$, the weak densities of all \EAS {\rm s} $(w + X)\cap Y$ are the same.
\end{corollary}
\section{Proof of Theorems \ref{thmq1}, \ref{thmq2}}
\label{proofs}
\subsection{Approximation by toric tongues.}
The proof of the theorems is based on
the use of approximation of algebraic variety by toric tongues.
The following notation is used below:

-\ $V=\re\mathcal T_G$, $N=\dim V$

-\ $K$ is a rational convex polyhedral cone in $V$

-\ $V_K\subset V$ is a subspace generated by $K$

-\ $\T_K\subset\T_G$ is the subtorus generated by the exponentials of the cone $K$

-\ $M\subset\T_G$ is a $k$-dimensional algebraic variety

-\ $\mathcal K$ is the tropical fan of the variety $M$

-\ $\mathcal K^m\subset\mathcal K$ is the subfan of cones of dimension $\leq k-m$

-\ $O_R(\mathcal K)=\{\tau\in\T_G\colon{\rm Re}\log\tau\not\in (\supp\mathcal K^1)_R\}$.
\par\smallskip
\begin{definition}\label{df_t}
The subset of $\T_G$
$$
t_{K,\tau}=\tau\exp(K+iV_K)\subset\tau\T_K
$$
is called \emph{a toric tongue}.
Cone $K$ and $\tau\in\T_G$ called respectively \emph{the base} and \emph{the shift}
of the tongue $t_{K,\tau}$.
\end{definition}
Let $\S,\L$ be subtori of $\T_G$,
 such that $\dim\S+\dim\L=N$ and $\#(\L\cap\S)=1$,
 and let $U$ be an open domain in some shift of the torus $\S$.
Consider a domain
 $$
U_\varepsilon=\{l\cdot t\:\vert\: l\in\L_\varepsilon, t\in U\},
 $$
 where $\L_\varepsilon$ is the $\varepsilon$-neighborhood of a unit in $\L$
and define the mapping
$\pi_\varepsilon\colon U_\varepsilon\to U$
 as
 $\pi_\varepsilon\colon l\cdot t\mapsto t.$
\begin{definition} \label{dfPert}
The subdomain $M\cap U_\varepsilon$ of the variety $M$
we call $\varepsilon$-perturbation of the domain of $U$,
 if the restriction  $\pi_\varepsilon$ to $M\cap U_\varepsilon$
 is a finite sheeted unramified covering of $U$.
\end{definition}
\begin{definition} \label{dfApproxSet}
Let $T(M)$ be a finite set of pairwise disjoint
$k$-dimensional toric tongues.
We call $T(M)$
\emph{an approximating set of tongues
$M$ with an approximating  $k$-dimensional fan} $\mathcal K$,
if the following is true

{\rm(i)} $\dim\mathcal K=k$ and
the set of tongue bases coincide with the set of
$k$-dimensional cones $K\in\mathcal K$

{\rm(ii)}
for any $\varepsilon$ there is $R$,
such that in $O_{R}(\mathcal K)$
the variety $M$ coincides with the union
$\varepsilon$-perturbations of all domains of the form
$O_{R-1}(\mathcal K)\cap t_{K,\tau}$,
where $t_{K,\tau}\in T(M)$.

We will call these $\varepsilon$-perturbations \emph{the perturbations of toric tongues}.
The degree of covering of a tongue from Definition \ref{dfPert} is called a \emph{weight of tongue}.
\end{definition}
\begin{theorem}\label{thmApprox}
A finely divided tropical fan of the algebraic variety $M$ is the approximating fan of $M$.
\end{theorem}
\subsection{Proof Theorem \ref{thmq1}}
Let $M\subset\T_G$ be a model of \EAS\ $X$,
$\codima X=n$.
Recall
that \EAS\ $X^g =\omega^{-1}_G(g^{-1}M)$,
where $\omega_G\colon\C^n\to\T_G$ is the mapping of standard winding,
called the toric shift \EAS\ $X$.
Let $\mathcal K$ and $T(M)=\{t_{K,\tau}\}$ be respectively
an approximating fan and a set of approximating tongues for $M$.
Then $T(g^{-1}(M))=\{t_{K,g^{-1}\tau}\}$ is the set of approximating tongues for $g^{-1}M$.
The approximating fan for $g^{-1}M$  equal $\mathcal K$.
Recall that $L_G={\rm Re}\log\:\omega_G(\C^n)\subset V$.
We give a sequence of simple statements,
leading to the proof of Theorem \ref{thmq1}.
\par\medskip
\textbf{(1)}
If $\dim K=k $ and the intersection $L_G\cap V_K$ is transversal,
then $\omega_G^{-1}\T_K$ is a $n$-dimensional lattice in the space $\im\C^n$.
\par\smallskip
\textbf{(2)} Consider the subset $\mathcal D(L_G,\mathcal K)\subset V$, consisting of points $v$,
such that the intersection $(v+L_G)\cap\supp\mathcal K$ is nonempty and not transversal.
Obviously, $\mathcal D(L_G,\mathcal K)$ belongs to the union of a finite set $\mathfrak I$
of proper subspaces of the space $V$.
Since, by construction, $\mathcal K^1\subset\mathcal D(L_G,\mathcal K)$, then

[\textbf{i}]\
if $v\in V\setminus\mathcal D(L_G,\mathcal K)_R$,
where $R$ is big enough then the affine subspace $v+L_G\subset V$ is located at a sufficiently large distance from the skeleton $\mathcal K^1$.
\par\smallskip
\textbf{(3)}
Let $\mathcal B$ be a connected component of $V\setminus\mathcal D(L_G,\mathcal K)$
and $v\in\mathcal B$.
Denote by $\mathcal K(\mathcal B)$ the set of cones $K\in\mathcal K$,
such that $(v+L_G)\cap K\ne\emptyset$.
Then

[\textbf{i}]\ the set is independent of the choice of $v\in\mathcal B$.

[\textbf{ii}]\ the set $\omega_G^{-1}\left(\bigcup_{K\in\mathcal K(\mathcal B)}\T_K\right)$ is a union
of a finite set of shifts of $n$-dimensional lattices
\begin{equation}\label{eqZ(B)}
  \mathcal Z(\mathcal B)=\{\Z_{1,\mathcal B},\ldots,\Z_{N_\mathcal B,\mathcal B}\}
\end{equation}
 in $\im\C^n$; см. \textbf{(1)}.
\par\smallskip
\textbf{(4)}
Let $T(M;\mathcal B)=\{t_{K,\tau}\in T(M)\colon\:K\in\mathcal K(\mathcal B)\}$ and
$$U_{R,\mathcal B}=\{g\in\T_G\colon\:{\rm Re}\log g\in V\setminus \mathcal D(L_G,\mathcal K)_R\}.$$
If $R$ is large enough, then for $g\in U_{R,\mathcal B}$ the following is true

[\textbf{i}]\ if $t_{K,\tau}\in T(M;\mathcal B)$,
then the intersection $g\omega_G(\C^n)\cap t_{K,\tau}$ is transversal and consists of a single point,
else $g\omega_G(\C^n)\cap t_{K,\tau}=\emptyset$.

[\textbf{ii}]\ \ $(g\omega_G)^{-1}\left(\bigcup_{t_{K,\tau}\in T(M;\mathcal B)}\:t_{K,\tau}\right)=
\bigcup_{1\leq i\leq N_\mathcal B}\: (z_i(g)+\Z_{i,\mathcal B})$,
where the functions $z_i\colon U_{R,\mathcal B}\to\C^n$ are continuous.
\par\smallskip
Now applying Theorem \ref{thmApprox},
we get that \EAS\ $(g\omega)^{-1}(M)$ is
a small perturbation of the union of shifted lattices from [\textbf{(4)}, \textbf{ii}].
The first assertion of the theorem \ref{thmq1} is proved.
The second statement is that the $n$-density of the union of lattices from (\ref{eqZ(B)}) independent of the connected component $\mathcal B$.

The first statement of the theorem is true for any
(including rational) subspace $L_G\subset\re\mathcal T_G$
(the irrationality property of the space $L_G$ was not used in the proof).
For a rational $ L_G $,
the second statement is equivalent to the balance condition
for the weights of a fan $\mathcal K$; see Section \ref{pullBack}.
Now the second statement follows from the continuous dependence of the density of the union of the lattices
from the set $\mathcal Z(\mathcal B)$ from the mapping of standart winding $\omega_G$.
\subsection{Proof of Theorem \ref{thmq2}}
Let $\mathcal P,\mathcal Q$ be tropicalizations of the models $P,Q$
of equidimensional \EAS s $X,Y$.
Recall that (this is proved in tropical geometry)
there is \rfm\ $B_\mathfrak I\subset V$ with a rational base $\mathfrak I$ (see the definition \ref{dfdRelFull}),
such that for sufficiently large $R$ the following is true:
if ${\rm Re}\log g\in B_\mathfrak I^R$,
 then the variety $P\cap gQ$ is equidimensional,
 and its tropicalization is equal to $\mathcal P\cdot\mathcal Q$.
Therefore, if
\begin{equation}\label{eqnuladno}
  {\rm Re}\log \omega_G(z)\in B_\mathfrak I^R,
\end{equation}
then \EAS\ $\,(z+X)\cap Y$ is equidimensional
and $\left(X\cdot Y\right)^{\rm trop}=X^{\rm trop}\cdot  Y^{\rm trop}$.
Hence, for $\codima X+\codima Y=n $ we get that for all such $z$ the weak density $d_w((z + X)\cap Y)$ is constant.

Consider a set subspaces $\mathcal I=\{I\subset\R^n\colon\:I=s_G^{-1}(J),\:J\in\mathfrak I\}.$
These subspaces are proper because the standard winding $\omega_G(\C^n)$ is everywhere dense.
Therefore, $B_\mathcal I\:$ is \rfm\ with base $\mathcal I$
and for all $z\in B_\mathcal I^R$
the condition (\ref{eqnuladno}) is satisfied.
The theorem is proved.
\par\medskip
\emph{Proof of Corollary} \ref{corthmq2}.
From tropical algebraic geometry it’s known that there exists an algebraic hypersurface $M\subset\T_G$,
such that for any $g\notin M$
the tropicalization of the variety $gP\cap Q$ is equal to $\mathcal P\cdot\mathcal Q$.
Using Theorem \ref{thmq2},
we obtain the statement for the exponential hypersurface $Z=\omega_G^{-1}M\subset\C^n$.
\section{Brief overview of tropical geometry essentials}
\label{pullBack}
Here we give a brief summary of basic tropical notions and,  then,
a description of  the construction for a pull back $s^*\mathcal K$
of tropical fan $\mathcal K\subset U$
with respect to the linear operator $s\colon V\to U$,
where $V$ is a vector space with the fixed orientation.
The properties of pull back mapping apply to the proof of our main results about \EAS s;  see Section \ref{EG}.
We start with several equivalent definitions of tropical variety.
\subsection{Tropical varieties}\label{tropVar}
\subsubsection{Definition of tropical variety}
Let $\mathcal K$ be a fan of cones in $N$-dimensional
vector space $V$, $\dim\mathcal K=k$.
For $K\in\mathcal K$ we denote by $V_K\subset V$ the subspace generated by the cone $K$.
The function $K\mapsto W(K)\in\bigwedge^qV^*$ on the set of oriented $p$-dimensional
cones in $\mathcal K$
is called a $p$-chain of degree $q$ if $W(K)$
changes its sign when the orientation of the cone $K$ changes.
As usual, we define a $(p-1)$-chain $d W$,
called the boundary of the $p$-chain $W$.
A chain $W$ is said to be closed if $d W = 0$.
\begin{definition}\label{dfWeightOld}
A $k$-dimensional fan $\mathcal K$ with a closed $k$-chain
of degree $N-k$ is said to be tropical if $V_K\subset\ker W (K)$
for any $k$-dimensional cone $K\in\mathcal K$, i.e.
$W(K)(v_1\wedge\ldots\wedge v_{N-k}) = 0$ for any $v_1\in V_K$.
 We call $W(K)$ the weight of the cone $K$.
 \end{definition}
Note that $W(K)$ can be seen as an even volume form in the space $V/V_K$.

Any partition of the tropical fan $\mathcal K$
with the weights inherited from $\mathcal K$
is also a tropical fan.
Two tropical fans are called equivalent,
if they have a common tropical partition.
The equivalence class of tropical fans is called
\emph{the tropical variety}.
For example, all tropical fans of dimension $N$ are equivalent.
\subsubsection{Euclidean and rational tropical varieties}
Let a Euclidean metric or an integer lattice be given in the space $V$.
Then we can consider the weight  $W$
as a numerical function $w$ on the set of cones
in the following way.
For a $k$-dimensional cone $K\in\mathcal K$ consider in the quotient space $V/V_K$
the corresponding quotient metric or quotient lattice.
Let $\Pi\subset V/V_K$ be a unit cube of the quotient metric
or a fundamental parallelotope of the quotient lattice.
We set $w(K)=W(K)(\xi_1\wedge\ldots\wedge\xi_k)$,
where $\xi_i$ are the sides of $\Pi$.
\begin{definition}\label{dfEuQ}
In the first and second cases we will talk
respectively about Euclidean and rational tropical fans
and tropical varieties.
Further we suppose,
that the weights of rational fans are rational.
\end{definition}
Note that when choosing a metric in the space $ V $
any tropical variety becomes Euclidean.

From the approximation theorem
(see Theorem \ref{thmApprox} in Section \ref{proofs})
it follows that to any $k$-dimensional algebraic subvariety $M$ of the torus $\T$ there corresponds a $k$-dimensional rational tropical fan in the space  $V=\re\mathcal T$.
The corresponding tropical variety is called \emph{the tropicalization} of $M$.
\subsubsection{Ring of tropical varieties}
Tropical varieties form a commutative graded ring;
see \cite{K03}.
The following two tropical theorems are conveniently formulated in the language of Euclidean and rational varieties.
Denote these algebras by ${\mathbb E}(V)$ and by $\Q(V)$ respectively.
Euclidean or rational fan of degree $N$,
those. point $0$ with real or rational weight,
we identify respectively with $\R$ or $\Q$.
Thus,  multiplication operation sets the pairings
$$
\mathcal P_k\colon{\mathbb E}_k(V)\times{\mathbb E}_{N-k}(V)\to\R,\,\,\,
\mathcal Q_k\colon\Q_k(V)\times\Q_{N-k}(V)\to\Q
$$
\begin{theorem}\label{thmPair}
The pairings $\mathcal P_k,\,\mathcal Q_k$ are non-degenerate.
\end{theorem}
Recall that to any convex polyhedron $\Delta$ there corresponds a fan of cones $\mathcal K_{\Delta,k}$,
consisting of cones,
dual to faces of $\Delta$ of dimension $\leq k$.
We supply cones of codimension $k$
weights equal areas of dual faces.
The fan $\mathcal K_{\Delta,k}$ is a Euclidean tropical fan. If the vertices $\Delta$ are integer,
and the face areas are measured using an integer lattice,
then $\mathcal K_{\Delta, k}$ is a rational tropical fan.
\begin{theorem}[see \cite{K03}]\label{thmSpan}
The algebras ${\mathbb E}(V)$, $\Q(V)$ are generated by elements
degrees $1$ of the form $\mathcal K_{\Delta, 1}$.
\end{theorem}
\subsection{Pull backs of tropical varieties}\label{pullBacks}
Further we assume
that the orientation of the kernel of a linear operator $s\colon V\to U$ is fixed.
If $s$ is surjective,  then the set
$s^{-1}\mathcal K=\{s^{-1}K\colon K\in\mathcal K\}$
form a fan of cones in the space $V$.
Let $W$ be the weight chain on the fan $\mathcal K$.
We agree orientation of any subspace $E\subset U$
with the orientation of the subspace $s^{-1}E\subset V$.
The map $s$ gives an isomorphism of the quotient spaces $U/U_K$ and $V/V_{s ^ {-1} K}$.
Therefore, the mentioned agreement allows considering
$s^*\left(W (K)\right)$  as the weight of the cone $s^{-1}K$.
For the surjective operator $s$,
denote by $s^*\mathcal K$
the fan of cones
$s^{-1}K$ with weights $s^*\left (W (K) \right)$.
\begin{definition}\label{dfBackSur}
If $s$ is surjective  then the tropical fan $s^*\mathcal K$
call a pull back of tropical fan $\mathcal K$.
\end{definition}
\noindent
Let $s$ be injective.
We identify its image $s(V)$ with the subspace $V\subset U$.
Consider $V$ as a tropical fan in the space $U$
with a single cone and any nonzero weight $T(V)$.
Then $\supp(\mathcal K\cdot V)\subset V$.
Let $L$ be the cone of maximum dimension of the fan $\mathcal K\cdot V$.
Then the weight of the cone $L$ is equal to $T(V)\wedge W(L)$,
where $W(L)$ is the uniquely defined volume form in the space $V/V_L$.
The cones of $\mathcal K\cdot V$ form a fan of cones $(\mathcal K\cdot V)_s\subset V$.
We equip this fan with the weight chain $W$ and consider it as a tropical fan in $V$.
\begin{definition}\label{dfBackInj}
If $s$ is injective  then we denote the tropical fan
$(\mathcal K\cdot V)_s\subset V$ by $s^*\mathcal K$
and call it the pull back of $\mathcal K$.
\end{definition}
\noindent
Define the pull back of tropical fan with respect to any linear operator 
$s\colon V\to U$ as follows.
Represent $s$ in the form $s=s_{\rm inj}\cdot s_{\rm surj}$, where the operator $s_{\rm surj}\colon V\to s(V) $ is surjective,  and the operator $s_ {\rm inj}\colon s(V)\to U$ is injective.
\begin{definition}\label{dfBack}
We put $s^*\mathcal K=(s^*_{\rm surj}\cdot s^*_{\rm inj})\: \mathcal K$.
\end{definition}
The main result on the pull backs of tropical varieties is as follows.
%
\begin{theorem}\label{thmBack}
{\rm(1)} For any $s\colon V\to U$ the pull back mapping
$\mathcal K\mapsto s^*\mathcal K$
is a ring isomorphism $s^*\colon\mathcal V(U)\to\mathcal V(V)$.

{\rm(2)} If $s=s_1\cdot s_2$ then $s^*=s_2^*\cdot s_1^*$.
%

{\rm(3)} Let $s'\colon U^*\to V^*$ be an operator adjoint to $s$,
and let $\Delta\subset U^*$ be a convex polyhedron.
Then $s^*\mathcal K_{\Delta,k}=\mathcal K_{s'\Delta,k}$.
\end{theorem}
Let $E\subset U$,
$E^*=U^*/E^\bot$,
where $E^\bot$ is an orthogonal complement of $E$.
Denote by $\Delta\subset U^*$ and $\pi\Delta$ respectively a convex polyhedron in $U^*$
and its image under the projection $\pi\colon U^*\to E^*$.
Now choose the Euclidean metric in $U$ and
consider the subspace $E$ as a Euclidean tropical fan
consisting of a single cone with a weight of $1$.
\begin{proposition}\label{prKDeltaProduct2}
Let $k\leq \dim E$.
Then
$\mathcal K_{\pi\Delta,k}=\iota^*\mathcal K_{\Delta,k}$,
where $\iota$ is an operator of embedding $E$ into $V$.
\end{proposition}
%
\begin{proof}
By definition \ref{dfBackInj} we have $\iota^*\mathcal K_{\Delta,k}=(\mathcal K_{\Delta,k}\cdot E)_\iota$.
If $k=\dim E=m$, then the statement follows from the tropical theorem BKK (see \cite{EKK20} Theorem 3.1.3)
for the system of equations
$\mathbf f_1=\ldots=\mathbf f_m=\mathbf g_1=\ldots=\mathbf g_{N-m}=0,$
where $\mathbf f_i$ are tropical polynomials with the common Newton polyhedron $\Delta$,
and $\mathbf g_1=\ldots=\mathbf g_{N-m}=0$ is a system of tropical equations of subspace $E$,
consisting of polynomials with the common Newton polyhedron of area $1$.
For $k<m$, the using of localization  of tropical varieties (see \cite{K03})
reduces the statement to the case $k=m$.
\end{proof}
Now we turn to the proof of Theorem \ref{thmBack}.

First, we prove statement (3).
If the operator $s$ is surjective,
then the adjoint operator $s'$ is injective and
the polyhedron $s'\Delta$ lies in the subspace $s'(U^*)\subset V^*$.
In this case,
by the definition of the fan $\mathcal K_{\Delta,K}$,
statement (3) follows from Definition \ref{dfBackSur}.
If the operator $s$ is injective,
then the adjoint operator $s'$ is surjective.
In this case,
according to the definition of \ref{dfBackInj},
statement (3) coincides  Proposition \ref{prKDeltaProduct2}
for $E=s(V)$.
If $s=s_{\rm inj}\cdot s_{\rm surj}$,
then $s'\Delta=s'_{\rm surj}s'_{\rm inj}\Delta $.
Therefore, according to Definition \ref{dfBack},
the required statement is reduced to the previous one.
Statement (3) is proved.

We pass to the proof of (2).
For tropical fans of the form $\mathcal K_{\Delta,k} $
statement (2) follows from (3).
Really,
$$
s_2^*s_1^*\mathcal K_{\Delta,k}=s_2^*\mathcal K_{s_1'\Delta,k}=\mathcal K_{s'_2s'_1\Delta,k}=\mathcal K_{(s_1s_2)'\Delta,k}=(s_1s_2)^*\mathcal K_{\Delta,k}.
$$
According to Theorm \ref{thmSpan},
any tropical variety $\mathcal K$ of degree $k$
can be represented as a finite sum
$\alpha_i\sum_\Delta\mathcal K_{\Delta,k}$.
The pull back mapping $s^*\colon\mathcal V(U)\to\mathcal V(V)$is linear.
Therefore (2) follows from (3).
\par\smallskip
To prove (1) we use the ring of convex polyhedra  ${\rm Pol}(E;E)$; see Definition \ref{dfPolRing}
and Theorem \ref{thmqIsoPol}.
\begin{theorem}\label{thmqPolCov}
Any linear operator $s'\colon U^*\to V^*$ can be extended to a ring homomorphism
${\rm Pol}(s)\colon{\rm Pol}(U^*,U^*)\to{\rm Pol}(V^*,V^*)$.
\end{theorem}
\begin{proof}
Recall the notation: $J(L)$ is an ideal in symmetric algebra $S(L)$ of a vector space $L$
and ${\rm Pol}(L,L)=S(L)/J(L)$
(see Definition \ref{dfPolRing}).
The operator $s'$ extends to a ring homomorphism
$\mathcal S(s)\colon\mathcal S(U^*) \to\mathcal S(V^*)$.
Thus, it remains to prove that $\mathcal S(s)(J_{U^*})\subset J_{V^*}$.

Known that the mapping $\Delta\mapsto\mathcal K_{\Delta,1} $
extends to the ring isomorphism $\mathcal N (L)\colon{\rm Pol} (L, L)\to\mathcal V (L)$;
see \cite{EKK20}.
According to Theorem \ref{thmBack} (3), the following diagram is commutative
\begin{equation}
 \begin{tikzcd}\label{cd3}
\mathcal S_k(U^*)\arrow{d}{\mathcal N(U)}\arrow{r}{\mathcal S(s)}&\mathcal S_k(V^*)\arrow{d}{\mathcal N(V)}
\\ \mathcal V(U)\arrow{r}{s^*}&\mathcal V(V)
\end{tikzcd}
\end{equation}
Let $\mathcal P\in J_{U ^ *}$.
Then, according to Theorem \ref{thmqIsoPol},
$\mathcal N(U)\:\mathcal P=0$.
The commutativity of the diagram (\ref{cd3}) implies
that $\mathcal N(V)\:\mathcal S(s) \:\mathcal P = 0$.
Again, applying Theorem \ref{thmqIsoPol} ,
we get that $\mathcal S(s)\:\mathcal P\in J_{V^*}$.
\end{proof}
Theorem \ref{thmBack} (1) follows from Theorem \ref{thmqPolCov}.
\end{document}